\documentclass{amsart}
\usepackage{amsfonts}
\usepackage{latexsym}
\usepackage{amssymb}
\usepackage{amsmath}
\usepackage{todonotes}

\newcommand{\R}{\mathbb R}

\newcommand{\N}{\mathbb N}

\newtheorem{thm}{Theorem}[section]
\newtheorem{lemma}[thm]{Lemma}
\newtheorem{proposition}[thm]{Proposition}
\newtheorem{cor}[thm]{Corollary}

\newtheorem{df}{Definition}[section]
\theoremstyle{remark}
\newtheorem*{rmk}{Remark}

\begin{document}


\title{John's ellipsoid and the integral ratio of a log-concave function}

\author[David Alonso, Bernardo Gonz\'alez, C. Hugo Jim\'enez, Rafael Villa]{David Alonso-Guti\'{e}rrez, Bernardo Gonz\'alez Merino, C. Hugo Jim\'enez, Rafael Villa}

\email{alonsod@unizar.es}
\email{bg.merino@tum.de}
\email{carloshugo@us.es}
\email{villa@us.es}

\address{Universidad de Zaragoza}
\address{Technische Universit\"at M\"unchen}
\address{Universidade Federal de Minas Gerais}
\address{Universidad de Sevilla}

\thanks{The first named author is partially supported by Spanish grants MTM2013-42105-P, P1·1B2014-35
projects and by IUMA. The second named author is partially supported by MINECO-FEDER project
reference MTM2012-34037, Spain, and by Consejer\'ia de Industria, Turismo, Empresa e Innovaci\'on de la CARM through Fundaci\'on
S\'eneca, Agencia de Ciencia y Tecnolog\'ia de la Regi\'on de Murcia, Programa de Formaci\'on Postdoctoral de Personal
Investigador.
The third named author is supported by CAPES and IMPA. The forth named author were partially supported by the Spanish Ministry of Science
and Innovation and FEDER Funds of European Union grant MTM2012-30748}
\date{\today}
\begin{abstract}
We extend the notion of John's ellipsoid to the setting of integrable log-concave functions. This will allow us to define the integral ratio of a log-concave function, which will extend the notion of volume ratio, and we will find the log-concave function maximizing the integral ratio. A reverse functional affine isoperimetric inequality will be given, written in terms of this integral ratio. This can be viewed as a stability version of the functional affine isoperimetric inequality.
\end{abstract}
\maketitle
\section{Introduction and notation}
Asymptotic geometric analysis is a rather new branch in mathematics, which comes from the interaction of convex geometry and local theory of Banach spaces. From its beginning, the research interests in this area have been focused in understanding the geometric properties of the unit balls of high-dimensional Banach spaces and their behavior as the dimension grows to infinity. The unit ball of a finite dimensional Banach space is a centrally symmetric convex body and some of these geometric properties include the study of sections and projections of convex bodies, which are also convex bodies. However, when the distribution of mass in a convex body is studied, a convex body $K$ is regarded as a probability space with the uniform probability on $K$ and then the projections of the measure on linear subspaces are not the uniform probability on a convex body anymore and  the class of convex bodies is left. Nevertheless, as a consequence of Brunn-Minkowski's inequality, we remain in the class of log-concave probabilities, which are the probability measures with a log-concave density with respect to the Lebesgue measure. It is natural then, to work in the more general setting of log-concave functions rather than in the setting of convex bodies and a big part of the research in the area has gone in the direction of extending results from convex bodies to log-concave functions (see, for instance, \cite{AKM}, \cite{FM}, \cite{AKSW}, \cite{KM}, \cite{C}, \cite{CF}), while many of the open problems in the field are nowadays stated in terms of log-concave functions rather than in terms of convex bodies.

In \cite{J} John proved that, among all the ellipsoids contained in a convex body $K$, there exists a unique ellipsoid $\mathcal{E}(K)$ with maximum volume. This ellipsoid is called the John's ellipsoid of $K$. Furthermore, he characterized the cases in which the John's ellipsoid of $K$ is the Euclidean ball $B_2^n$. This characterization, together with Brascamp-Lieb inequality \cite{BL}, led to many important results in the theory of convex bodies, showing that, among centrally symmetric convex bodies, the cube is an extremal convex body for many geometric parameters like the Banach-Mazur distance to the Euclidean ball, the volume ratio, the mean width, or the mean width of the polar body, see \cite{B}, \cite{SS}, \cite{Ba}. The non-symmetric version of these problems has also been studied, see for instance \cite{S}, \cite{Le}, \cite{Pa}, \cite{JN}, \cite{Sch1}.

A function $f:\R^n\to\R$ is said to be log-concave if it is of the form $f(x)=e^{-v(x)}$, with $v:\R^n\to (-\infty,+\infty]$ a convex function. Note that log-concave functions are continuous on their support and, since convex functions are differentiable almost everywhere, then so are log-concave functions. In this paper we will extend John's theorem to the context of log-concave functions. We will consider ellipsoidal functions (we will sometimes simply call them ellipsoids), which will be functions of the form
$$\mathcal{E}^a(x)=a\chi_{\mathcal{E}}(x),$$
with $a$ a positive constant and $\chi_\mathcal{E}$ the characteristic function of an ellipsoid $\mathcal{E}$, {\it i.e.}, an affine image of the Euclidean ball ($\mathcal{E}= c+TB_2^n$ with $c\in\R^n$ and $T\in GL(n)$, the set of linear matrices with non-zero determinant). The determinant of a matrix $T$ will be denoted by $|T|$. The volume of a convex body $K$ will also be denoted by $|K|$. The trace of $T$ will be denoted by $\textrm{tr}(T)$.

Given a log-concave function $f:\R^n\to\R$, we will say that an ellipsoid $\mathcal{E}^a$ is contained in $f$ if for every $x\in\R^n$, $\mathcal{E}^a(x)\leq f(x)$. Notice that if $\mathcal{E}^a\leq f$, then necessarily $0<a\leq\Vert f\Vert_\infty$ and that for any $t\in(0,1]$
$$
\mathcal{E}^{t\Vert f\Vert_\infty} \leq f
$$
if and only if the ellipsoid $\mathcal{E}$ is contained in the convex body
$$
K_t(f)=\{x\in\R^n\,:\,f(x)\geq t\Vert f\Vert_\infty\}.
$$
If $f=\chi_K(x)$ is the characteristic function of a convex body $K$, then an ellipsoid $\mathcal{E}$ is contained in $K$ if and only if $\mathcal{E}^t\leq f$ for any $t\in(0,1]$. In Section \ref{ExistenceUniquenessSection} we will show the following:
\begin{thm}\label{ExistenceUniquenessEllipsoid}
Let $f:\R^n\to\R$ be an integrable log-concave function. There exists a unique ellipsoid $\mathcal{E}(f)=\mathcal{E}^{t_0\Vert f\Vert_\infty}$ for some $t_0\in(0,1]$, such that
\begin{itemize}
\item $\mathcal{E}(f)\leq f$
\item $\displaystyle{\int_{\R^n}\mathcal{E}(f)(x)dx=\max\left\{\int_{\R^n}\mathcal{E}^{a}(x)dx\,:\,\mathcal{E}^{a}\leq f\right\}.}$
\end{itemize}
We will call this ellipsoid the John's ellipsoid of $f$.
\end{thm}

The existence and uniqueness of the John's ellipsoid of an integrable log-concave function $f$ will allow us to define the integral ratio of $f$:
\begin{df}
Let $f:\R^n\to\R$ be an integrable log-concave function and $\mathcal{E}(f)$ its John's ellipsoid. We define the integral ratio of $f$:
$$
I. rat(f)=\left(\frac{\int_{\R^n}f(x)dx}{\int_{\R^n}\mathcal{E}(f)(x)dx}\right)^\frac{1}{n}.
$$
\end{df}
\begin{rmk}
This quantity is affine invariant, i.e., $I. rat(f\circ T)=I. rat(f)$ for any affine map $T$. When $f=\chi_K$ is the characteristic function of a convex body then $I. rat(f)=v.rat(K)$, the volume ratio of $K$ (Recall that $v.rat(K)=\left(\frac{|K|}{|\mathcal{E}(K)|}\right)^\frac{1}{n}$, where $\mathcal{E}(K)$ is the John's ellipsoid of $K$).
\end{rmk}
In Section \ref{CharacterizationJohnsPosition} we will give an upper bound for the integral ratio of log-concave functions, finding the functions that maximize it. Namely, denoting by $\Delta_n$ and $B_\infty^n$ the regular simplex centered at the origin and the unit cube in $\R^n$, and by $\Vert\cdot\Vert_K$ the gauge function associated to a convex body $K$ containing the origin, which is defined as
$$
\Vert x \Vert_K\inf\{\lambda>0\,:\,x\in\lambda K\},
$$
we will prove the following
\begin{thm}\label{TheoremMaximizer}
Let $f:\R^n\to\R$ be an integrable log-concave function. Then,
$$
I.rat(f)\leq I.rat(g_c),
$$
where
$g_c(x)=e^{-\Vert x\Vert_{\Delta_n-c}}$ for any $c\in\Delta^n$. Furthermore, there is equality if and only if $\frac{f}{\Vert f\Vert_\infty}=g_c\circ T$ for some affine map $T$ and some $c\in\Delta^n$. If we assume $f$ to be even, then
$$
I.rat(f)\leq I.rat(g),
$$
where
$g(x)=e^{-\Vert x\Vert_{B_\infty^n}}$. with equality if and only if $\frac{f}{\Vert f\Vert_\infty}=g\circ T$ for some linear map $T\in GL(n)$.
\end{thm}
In order to do so we will prove a characterization of the situation in which the John's ellipsoid of an integrable log-concave function is $\mathcal{E}(f)=(B_2^n)^{t_0\Vert f\Vert_\infty}$. In such case we will say that a log-concave function is in John's position.

The isoperimetric inequality states that for any convex body $K$ the quantity $\frac{|\partial K|}{|K|^{\frac{n-1}{n}}}$ is minimized when $K$ is a Euclidean ball. This inequality cannot be reversed in general. However, in \cite{B}, it was shown that for any symmetric convex body $K$, there exists an affine image $TK$ such that the quotient $\frac{|\partial TK|}{|TK|^{\frac{n-1}{n}}}$ is bounded above by the corresponding quantity for the cube $B_\infty^n$. If we do not impose symmetry then the regular simplex is the maximizer. This linear image is the one such that $TK$ is in John's position, {\it i.e.}, the maximum volume ellipsoid contained in $K$ is the Euclidean ball. The quantity studied in the isoperimetric inequality is not affine invariant but in \cite{P}, a stronger affine version of the isoperimetric inequality was established. Namely, it was shown that for any convex body $K$
$$
|K|^\frac{n-1}{n}|\Pi^*(K)|^\frac{1}{n}\leq |B_2^n|^\frac{n-1}{n}|\Pi^*(B_2^n)|^\frac{1}{n},
$$
where $\Pi^*(K)$, which is called the polar projection body of $K$, is the unit ball of the norm $\Vert x\Vert_{\Pi^*(K)}=|x||P_{x^\perp}K|$, being $P_{x^\perp}K$ the projection of $K$ onto the hyperplane orthogonal to $x$. This inequality is known as Petty's projection inequality and there is equality in it if and only if $K$ is an ellipsoid. Furthermore, following the idea in the proof of the reverse isoperimetric inequality, a stability version of it was given in \cite{A}, showing that for any convex body $K$
\begin{equation}\label{ReversePetty}
|K|^\frac{n-1}{n}|\Pi^*(K)|^\frac{1}{n}\geq\frac{1}{v.rat(K)}|B_2^n|^\frac{n-1}{n}|\Pi^*(B_2^n)|^\frac{1}{n}.
\end{equation}

The isoperimetric inequality and Petty's projection inequality have their functional extensions. Namely, Sobolev's inequality, which states that for any function $f$ in the Sobolev space
$$
W^{1,1}(\R^n)=\left\{f\in L^1(\R^n)\,:\,\frac{\partial f}{\partial x_i}\in L^1(\R^n)\quad \forall i\right\}
$$
we have
$$
\Vert|\nabla f|\Vert_1\geq n|B_2^n|^\frac{1}{n}\Vert f\Vert_{\frac{n}{n-1}},
$$
and the affine Sobolev's inequality, proved in \cite{Z}, which states that
\begin{equation}\label{AffineSobolev}
\Vert f\Vert_{\frac{n}{n-1}}|\Pi^*(f)|^\frac{1}{n}\leq\frac{|B_2^n|}{2|B_2^{n-1}|},
\end{equation}
where $\Pi^*(f)$ is the unit ball of the norm
$$
\Vert x\Vert_{\Pi^*(f)}=\int_{\R^n}|\langle\nabla f(y),x\rangle| dy.
$$
We would like to recall here the fact that $W^{1,1}(\R^n)$ is the closure of $\mathcal{C}_{00}^1$, the space of $\mathcal{C}^1$ functions with compact support, \cite{M}.
These inequalities are actually equivalent to their geometric counterparts.

In Section \ref{ReverseSobolev} we will follow the same ideas to obtain functional versions of the reverse isoperimetric inequality and a stability version of the affine Sobolev inequality. We will prove the following extension of \eqref{ReversePetty}, which is a reverse form of \eqref{AffineSobolev} in the class of log-concave functions.
\begin{thm}\label{FunctionalReversePetty}
Let $f\in W^{1,1}(\R^n)$ be a log-concave function. Then
$$
\frac{\Vert f\Vert_{\frac{n}{n-1}}|\Pi^*(f)|^\frac{1}{n}}{\left(\frac{|B_2^n|}{2|B_2^{n-1}|}\right)}\geq\frac{1}{e^{\frac{\int_{\R^n}f(x)\log\left(\frac{f(x)}{\Vert f\Vert_\infty}\right)^\frac{1}{n}dx}{\int_{\R^n}f(x)dx}}\Vert f\Vert_\infty^\frac{1}{n}\left(\frac{\int_{\R^n}f(x)dx}{\int_{\R^n}f^\frac{n}{n-1}(x)dx}\right)^\frac{n-1}{n}I.rat(f)}.
$$
\end{thm}
\begin{rmk}
By (\ref{AffineSobolev}) the left-hand side term is bounded above by 1. This lower bound is affine invariant, and if $f=\chi_K$ is the characteristic function of a convex body, then we recover inequality \eqref{ReversePetty}.
\end{rmk}
\begin{rmk}
Let us note that if $\int_{\R^n} f(x)dx=1$ the previous inequality turns into

$$e^{\frac{-1}{n}\int f(x)\log f(x)}dx\leq\frac{I.rat(f)|\Pi^*(f)|^\frac{1}{n}}{\left(\frac{|B_2^n|}{2|B_2^{n-1}|}\right)},$$
which along with the affine Sobolev inequality (\ref{AffineSobolev}) provides us with a bound for the power entropy of $f$ of the following form

$$H(f):=e^{\frac{-2}{n}\int_{\R^n} f(x)\log f(x)dx}\leq \left(\frac{I.rat(f)}{\|f\|_{\frac{n}{n-1}}}\right)^2.$$

For other recently studied connections between Information theory and convex geometry we refer to  \cite{BM1}, \cite{BM2} and references therein.
\end{rmk}
Let us introduce some more notation: If $K$ is a convex body, $r(K)$ will denote its inner radius, {\it i.e.}, the radius of the largest centered Euclidean ball contained in it. For a set $A\subseteq\R^n$, the positive hull of $A$ is the convex cone
$$
\textrm{pos}A=\left\{\sum_{i=1}^n\lambda_ix_i\,:\,\lambda_i>0, x_i\in A, n\in\N\right\}.
$$
Given a convex set $E\subseteq\R^n$ and $x\in\partial E$, the boundary of $E$, the normal cone of $E$ at $x$ is defined as
$$
N(E,x)=\{u\in\R^n\,:\,\langle z-x,u\rangle\leq 0,\, \forall z\in E\}.
$$
The support cone of $E$ at $x$ is the cone
$$
S(E,x)=\textrm{cl}\bigcup_{\lambda>0}\lambda(E-x).
$$
The following polarity relation holds:
$$
N(E,x)^*=S(E,x),
$$
where the polarity relation is the polarity of convex cones
$$
C^*=\{y\in\R^n\,:\,\langle y,x\rangle\leq 0,\,\forall x\in C\}.
$$
If $H$ is an affine subspace through $x$, then the normal cone to $E\cap H$ at $x$, relative to the subspace $H$ is
$$
N_H(E\cap H,x)=\{u\in H_0\,:\,\langle z-x,u\rangle\leq 0,\, \forall z\in E\cap H\},
$$
where $H_0$ is the linear subspace parallel to $H$. The similar duality holds
$$
N_H(E\cap H,x)^{*H_0}= S(E\cap H,x),.
$$
where the duality is taken with respect to the linear subspace $H_0$. It happens  that
\begin{equation}\label{DualityConesSubspaces}
N_H(E\cap H,x)=P_{H_0}N(E,x),
\end{equation}
where $P_{H_0}$ denotes the orthogonal projection onto the linear subspace $H_0$. We refer the reader to \cite{Sch} for these and other known facts on convex cones.

For any function $f:\R^n\to\R$ and any $\varepsilon>0$, we will denote $f_\varepsilon$ the function given by
$$
f_\varepsilon(x)= f\left(\frac{x}{\varepsilon}\right)^\varepsilon.
$$
If $f$ and $g$ are two log-concave functions, then their Asplund product is the log-concave function
$$
f\star g(z)=\max_{z=x+y}f(x)g(y)=\max_{y\in\R^n}f(z-y)g(y).
$$
\section{John's ellipsoid of a log-concave function}\label{ExistenceUniquenessSection}

In this section we show the existence and uniqueness of the John's ellipsoid of an integrable log-concave function and show that the integral ratio of a function is an affine invariant.

For any ellipsoid $\mathcal{E}^a$, its integral is $a|\mathcal{E}|$. Since for any $t\in(0,1]$ the convex body $K_t(f)$ has a unique maximum volume ellipsoid $\mathcal{E}_t(f)=\mathcal{E}(K_t(f))$, then
$$
\max\left\{\int_{\R^n}\mathcal{E}^{a}(x)dx\,:\,\mathcal{E}^{a}\leq f\right\}=\max_{t\in(0,1]}\phi_f(t),
$$
where $$\phi_f(t)=t\Vert f\Vert_\infty|\mathcal{E}_t(f)|.$$
Thus, in order to prove Theorem \ref{ExistenceUniquenessEllipsoid} we need to prove that the function $\phi_f(t)$ attains a unique maximum in the interval $(0,1]$ at some point $t_0$. Then the ellipsoid $\mathcal{E}(f)$ will be the function
$$
\mathcal{E}(f)(x)=t_0\Vert f\Vert_\infty\chi_{\mathcal{E}_{t_0}(f)}(x)=\left(\mathcal{E}_{t_0}(f)\right)^{t_0\Vert f\Vert_\infty}(x),
$$
where $\mathcal{E}_{t_0}(f)$ is the John's ellipsoid of the convex body $K_{t_0}(f)$. If $f=\chi_K$ with $K$ a convex body then the John's ellipsoid of $f$ will be the characteristic function of the John's ellipsoid of $K$ $\mathcal{E}(K)^1=\chi_{\mathcal{E}(K)}$. We will prove that $\phi_f$ attains a unique maximum in the interval $(0,1]$. First we prove the following:

\begin{lemma}\label{concavityPhi}
Let $f$ be a log-concave function and let $\phi_f:(0,1]\to\R$ defined as before. For any $t_0,t_1\in(0,1]$ and any $\lambda\in[0,1]$
$$
\phi_f(t_0^{1-\lambda}t_1^\lambda)\geq\phi_f(t_0)^{1-\lambda}\phi_f(t_1)^\lambda.
$$
\end{lemma}
\begin{proof}
Since $f$ is log-concave
\begin{eqnarray*}
\{x\in\R^n\,:\,f(x)\geq t_0^{1-\lambda}t_1^\lambda\Vert f\Vert_\infty\}
&\supseteq&
(1-\lambda)\{x\in\R^n\,:\,f(x)\geq t_0\Vert f\Vert_\infty\}
\cr
&+&
\lambda\{x\in\R^n\,:\,f(x)\geq t_1\Vert f\Vert_\infty\}.
\end{eqnarray*}
Thus, if $\mathcal{E}_{t_i}(f)=a_i+T_iB_2^n$ with $T_i$ a symmetric positive definite matrix, $i=0,1$, then
\begin{eqnarray}
\label{eq:Econvex}
\{x\in\R^n\,:\,f(x)\geq t_0^{1-\lambda}t_1^\lambda\Vert f\Vert_\infty\}&\supseteq&(1-\lambda)\mathcal{E}_{t_0}(f)+\lambda\mathcal{E}_{t_1}(f)\cr
&=&(1-\lambda)a_0+\lambda a_1+(1-\lambda)T_{t_0}B_2^n+\lambda T_{t_1}B_2^n\cr
&\supseteq&(1-\lambda)a_0+\lambda a_1+((1-\lambda)T_{t_0}+\lambda T_{t_1})B_2^n.
\end{eqnarray}
Taking volumes, since by Minkowski's determinant inequality, for any two symmetric positive definite matrices $A,B$ we have that $|A+B|^\frac{1}{n}\geq|A|^\frac{1}{n}+|B|^\frac{1}{n}$ with equality if and only if $B=s A$ for some $s>0$, we obtain
\begin{eqnarray}\label{eq:Econvex2}
|\mathcal{E}_{t_0^{1-\lambda}t_1^\lambda}(f)|^\frac{1}{n}
&\geq&
|(1-\lambda)T_0+\lambda T_1|^\frac{1}{n}|B_2^n|^\frac{1}{n}\cr
&\geq&
((1-\lambda)|T_0|^\frac{1}{n}+\lambda|T_1|^\frac{1}{n})|B_2^n|^\frac{1}{n}
\cr
&=&(1-\lambda)|\mathcal{E}_{t_0}(f)|^\frac{1}{n}+\lambda|\mathcal{E}_{t_1}(f)|^\frac{1}{n}
\cr
&\geq&
|\mathcal{E}_{t_0}(f)|^\frac{1-\lambda}{n}|\mathcal{E}_{t_1}(f)|^\frac{\lambda}{n},
\end{eqnarray}
where the last inequality is the arithmetic-geometric mean inequality. Consequently,
$$
|\mathcal{E}_{t_0^{1-\lambda}t_1^\lambda}(f)|\geq|\mathcal{E}_{t_0}(f)|^{1-\lambda}|\mathcal{E}_{t_1}(f)|^\lambda
$$
and multiplying by $t_0^{1-\lambda}t_1^\lambda\Vert f\Vert_\infty$
$$
\phi_f(t_0^{1-\lambda}t_1^\lambda)\geq\phi_f(t_0)^{1-\lambda}\phi_f(t_1)^\lambda.
$$
\end{proof}

Now, Theorem \ref{ExistenceUniquenessEllipsoid} will be a consequence of the following
\begin{lemma}
Let $f:\R^n\to[0,+\infty)$ be an integrable log-concave function and let $\phi_f:(0,1]\to\R$ defined as before. Then $\phi_f$ is continuous in $(0,1]$ and
$$
\lim_{t\to0^+}\phi_f(t)=0.
$$
Consequently $\phi_f$ attains its maximum value at some $t_0\in(0,1]$. Furthermore, such $t_0$ is unique.
\begin{proof}
In order to show the continuity of $\phi_f$ in $(0,1]$ from the right it is enough to show that for any $t_0\in(0,1)$ , if $d_{t_0}(\varepsilon)$ is the smallest number such that
$$
K_{t_0+\varepsilon}(f)\subseteq K_{t_0}(f)\subseteq d_{t_0}(\varepsilon)K_{t_0+\varepsilon}(f),
$$
then $\lim_{\varepsilon\to 0^+}d_{t_0}(\varepsilon)=1$, since then
$$
|\mathcal{E}_{t_0+\varepsilon}(f)|\leq |\mathcal{E}_{t_0}(f)|\leq d_{t_0}(\varepsilon)^n|\mathcal{E}_{t_0+\varepsilon}(f)|
$$
and consequently
$$
\lim_{\varepsilon\to 0^+}|\mathcal{E}_{t_0+\varepsilon}(f)|\leq |\mathcal{E}_{t_0}(f)|
$$
and
$$
\lim_{\varepsilon\to 0^+}|\mathcal{E}_{t_0+\varepsilon}(f)|\geq\lim_{\varepsilon\to 0^+}\frac{1}{d_{t_0}(\varepsilon)^n}|\mathcal{E}_{t_0}(f)|=|\mathcal{E}_{t_0}(f)|.
$$
Let us see then that $\lim_{\varepsilon\to 0^+}d_{t_0}(\varepsilon)=1$.

Notice that for every $x\in K_{t_0}(f)\backslash\lim_{\varepsilon\to0^+}K_{t_0+\varepsilon}(f)$ we have that $f(x)=t_0\Vert f\Vert_\infty$. Assume that $\lim_{\varepsilon\to0^+}d_{t_0}(\varepsilon)$ is not 1. Then there exists a segment $[x_0,x_1]$  and a point $c\in(x_0,x_1)$ such that $[x_0,c]$ is contained in $K_{t_0}(f)\backslash\lim_{\varepsilon\to0^+}K_{t_0+\varepsilon}(f)$ and $x_1\in K_{t_1}(f)$ for some $t_1>t_0$. Then, since $f$ is log-concave $f(c)>f(x_0)=t_0$, which contradicts the fact that $[x_0,c]$  is contained in $K_{t_0}(f)\backslash\lim_{\varepsilon\to0^+}K_{t_0+\varepsilon}(f)$.

A similar argument
 proves that $\phi_f$ is continuous from the left.

Let us now prove that $\lim_{t\to0^+}\phi_f(t)=0$. Let $\varepsilon>0$. Since $f$ is integrable, we can find $R(\varepsilon)$ big enough such that
$$
\int_{\R^n\backslash R(\varepsilon)B_2^n}f(x)dx<\frac{\varepsilon}{2}.
$$
Now, for any $t<\frac{\varepsilon}{2\Vert f\Vert_\infty|R(\varepsilon)B_2^n|}$ we have that
\begin{eqnarray*}
t\Vert f\Vert_\infty|K_t(f)|&=&t\Vert f\Vert_\infty|K_t(f)\cap R(\varepsilon)B_2^n| +t\Vert f\Vert_\infty|K_t(f)\backslash
R(\varepsilon)B_2^n|\cr &<&t\Vert f\Vert_\infty|R(\varepsilon)B_2^n|
+\int_{K_t(f)\backslash R(\varepsilon)B_2^n}f(x)dx\cr
&<&\frac{\varepsilon}{2}+\int_{\R^n\backslash
R(\varepsilon)B_2^n}f(x)dx\cr
&<&\frac{\varepsilon}{2}+\frac{\varepsilon}{2}=\varepsilon.
\end{eqnarray*}
Then,
$$
0\leq\lim_{t\to0^+}t\Vert f\Vert_\infty|\mathcal{E}_t(f)|\leq\lim_{t\to0^+}t\Vert f\Vert_\infty|K_t(f)|=0
$$
and so
$$
\lim_{t\to0^+}\phi_f(t)=0.
$$
Consequently $\phi_f$ attains its maximum for some $t_0\in(0,1]$. Let us prove that such $t_0$ is unique. Assume that there exist two different $t_1<t_2$ at which $\phi_f$ attains its maximum. Then, by Lemma \ref{concavityPhi} for any $\lambda\in[0,1]$
$$
\phi_f(t_1^{1-\lambda}t_2^\lambda)=\phi_f(t_1)^{1-\lambda}\phi_f(t_2)^\lambda.
$$
Thus, for any $\lambda\in[0,1]$
$$
|\mathcal{E}_{t_1^{1-\lambda}t_2^\lambda}(f)|^\frac{1}{n}=|\mathcal{E}_{t_1}(f)|^\frac{1-\lambda}{n}|\mathcal{E}_{t_2}(f)|^\frac{\lambda}{n}
$$
and all the inequalities in \eqref{eq:Econvex2} are equalities. This implies that $T_{t_2}$ is a multiple of $T_{t_1}$ and so the ellipsoids $\mathcal{E}_{t_1}$ and $\mathcal{E}_{t_2}$ are homothetic. Besides, since there is equality in the arithmetic-geometric mean inequality both ellipsoids have the same volume and, $\mathcal{E}_{t_2}$ is a translate of $\mathcal{E}_{t_1}$. Thus, for any $\lambda\in[0,1]$ we have
$$
|\mathcal{E}_{t_1^{1-\lambda}t_2^\lambda}(f)|=|\mathcal{E}_{t_1}(f)|.
$$
But then $\phi_f(t_2)>\phi_f(t_1)$, which contradicts the assumption of the maximum being attained at two different points.
\end{proof}
\end{lemma}

Now that we have established the existence and uniqueness of the John's ellipsoid of an integrable log-concave function $f$, we can define the integral ratio of $f$ as
$$
I. rat(f)=\left(\frac{\int_{\R^n}f(x)dx}{\int_{\R^n}\mathcal{E}(f)(x)dx}\right)^\frac{1}{n}=\left(\frac{\int_{\R^n}f(x)dx}{\max_{t\in(0,1]}\phi_f(t)}\right)^\frac{1}{n}.
$$

The integral ratio of a function is an affine invariant, {\it i.e.}, for any affine map $T$ we have that $I.rat(f\circ T)=I.rat(f)$. This is a consequence of the following lemma.
\begin{lemma}\label{LinearMapsAndEllipsoids}
Let $f:\R^n\to\R$ be an integrable log-concave function and let $T$ be an affine map. Then for any $t\in(0,1]$
$$
\mathcal{E}_t(f\circ T^{-1})=T\mathcal{E}_t(f).
$$
As a consequence
$$
\phi_{f\circ T^{-1}}(t)=|T|\phi_f(t),
$$
the maximum of $\phi_{f\circ T^{-1}}$ and $\phi_f$ is attained for the same $t_0$, and
$$
\mathcal{E}(f\circ T^{-1})=\mathcal{E}(f)\circ T^{-1}.
$$
\end{lemma}
\begin{proof}
Notice that
$$
K_t(f\circ T^{-1})=\{x\in\R^n\,:\,f(T^{-1}x)\geq t\Vert f\Vert_\infty\}=T\{x\in\R^n\,:\,f(x)\geq t\Vert f\Vert_\infty\}=TK_t(f).
$$
Consequently
$$
\mathcal{E}_t(f\circ T^{-1})=T\mathcal{E}_t(f).
$$
\end{proof}

\section{John's position of a log-concave function and maximal integral ratio}\label{CharacterizationJohnsPosition}

A log-concave function will be said to be in John's position if $\mathcal{E}(f)=(B_2^n)^{t_0\Vert f\Vert_\infty}$ for some $t_0\in(0,1]$. By Lemma \ref{LinearMapsAndEllipsoids}, for any log-concave integrable function there exists an affine map $T$ such that $f\circ T$ is in John's position. In this section we will give a characterization for a function to be in John's position. As a consequence we will obtain an estimate for the function $\phi_f(t)$ that will allow us to give an upper bound for the integral ratio of any integrable log-concave function. We will follow the ideas in \cite{GS} and prove the following

\begin{thm}\label{FunctionalJohn}
Let $f:\R^n\to\R$ be an even integrable log-concave function and $t_0\in(0,1]$. Assume that $(B_2^n)^{t_0\Vert f\Vert_\infty}\leq f$. Then the following are equivalent:
\begin{enumerate}
\item[1.] $\mathcal{E}(f)=(B_2^n)^{t_0\Vert f\Vert_\infty}$.
\item[2.] There exist
 \begin{itemize}
 \item  $\{u_{ij}\}\subseteq \partial K_{t_0}\cap S^{n-1}$, $1\leq i\leq m, 1\leq j\leq m^\prime(i)$
 \item  $\{\lambda_i\}_{i=1}^m$, $\{\mu_{ij}\}_{j=1}^{m^\prime(i)}$, with $\lambda_i,\mu_{ij}>0$ and
 \item  $\{\alpha_i\}_{i=1}^m$, with $\alpha_i\in\left[t_0\left.\frac{d}{dt^+}r(K_t)\right|_{t=t_0},t_0\left.\frac{d}{dt^-}r(K_t)\right|_{t=t_0}\right]$
 \end{itemize}
with $1\leq m\leq 1+\frac{n(n+1)}{2}$ and $n\leq m^\prime(i)\leq \frac{n(n+1)}{2}$ for any $1\leq i\leq m$, such that
$$
\sum_{i=1}^m\lambda_i\alpha_i=-1,
$$
$$
I_{n}=\sum_{i=1}^m\sum_{j=1}^{m^\prime(i)}\lambda_i\mu_{ij} u_{ij}\otimes u_{ij},
$$
and for any $t\in (0,1]$, and any $1\leq i\leq m$ if $TB_2^n\subseteq K_t(f)$
$$
-\alpha_i\log\left(\frac{t}{t_0}\right)+\sum_{j=1}^{m^\prime(i)}\mu_{ij}\langle Tu_{ij},u_{ij}\rangle\leq \sum_{j=1}^{m^\prime(i)}\mu_{ij}.
$$
\end{enumerate}
\end{thm}
\begin{proof}
Any ellipsoid $\mathcal{E}\subseteq\R^n$ is $\mathcal{E}=TB_2^n$ with $T$ a symmetric positive definite matrix. For any symmetric positive definite matrix $T$, we will call $x_T\in\R^{n^2}$ the vector
$$
x_T=\left(t_{11},\dots,t_{1n},t_{21},\dots,t_{2n},t_{n1},\dots,t_{nn}\right)^t.
$$
Notice that the set $C_{\textrm{spd}}=\{x_{T}\,:\,T\textrm{ symmetric positive definite}\}$ is a convex cone contained in a linear subspace $L\subseteq\R^{n^2}$ of dimension $\frac{n(n+1)}{2}$. We will consider the following two sets in $\R\times L$:
$$
E=\{(s,x_T)\in[0,+\infty]\times L\,:\,T\in C_{\textrm{spd}}, (TB_2^n)^{e^{-s}\Vert f\Vert_\infty}\leq f\}
$$
and
$$
C_1=\{(s,x_T)\in[0,+\infty]\times L:T\in C_{\textrm{spd}}, e^{-s}|T|\geq t_0\}.
$$
First of all, notice that both of them are convex. In order to show that $E$ is convex set let $(s_i,x_{T_i})\in E$, $i=1,2$. From the definition of $E$, this means that
$$
K_{e^{-s_i}}\supseteq T_iB_2^n.
$$
By \eqref{eq:Econvex} we get that
\[
K_{e^{-(1-\lambda)s_1-\lambda s_2}}\supseteq ((1-\lambda)T_1+\lambda T_2)B^n_2
\]
from which $((1-\lambda)s_1+\lambda s_2,x_{(1-\lambda)T_1+\lambda T_2})\in E$.

In order to see that $C_1$ is convex let $(s_i,x_{T_i})\in C_1$, $i=1,2$.
Then $e^{-s_i}|T_i|\geq t_0$, $i=1,2$. Minkowski's determinant inequality and the arithmetic-geometric mean imply that
\begin{equation*}
\begin{split}
e^{-\frac{(1-\lambda)s_1+\lambda s_2}{n}}|(1-\lambda)T_1+\lambda T_2|^\frac{1}{n}&\geq\,e^{-\frac{(1-\lambda)s_1+\lambda s_2}{n}}((1-\lambda)|T_1|^{\frac{1}{n}}+\lambda|T_2|^{\frac{1}{n}})\\
&\geq e^{-\frac{(1-\lambda)s_1+\lambda s_2}{n}}|T_1|^{\frac{(1-\lambda)}{n}}|T_2|^{\frac{\lambda}{n}}\\
&=\left(e^{-s_1}|T_1|\right)^\frac{1-\lambda}{n}\left(e^{-s_2}|T_2|\right)^\frac{\lambda}{n}\\
&\geq t_0^\frac{1}{n},
\end{split}
\end{equation*}
from which we conclude $((1-\lambda)s_1+\lambda s_2,(1-\lambda)T_1+\lambda T_2)\in C_1$.

Second, notice that if $s\neq 0$ and $e^{-s}|T|=t_0$, then the point $(s,x_T)$ belongs to the boundary of $C_1$, which is smooth around it. Then there exists a unique supporting hyperplane of $C_1$ at $(s,x_T)$. Since the function $g(s)=e^\frac{s}{n}$ is convex, its graph is above its tangent line at the point $(-\log t_0, t_0^{-\frac{1}{n}})$. Thus, for any $(s,x_T)\in C_1$,
\begin{eqnarray*}
\langle(-1,x_{I_n}), (s,x_T) \rangle&=&-s+\textrm{tr}(T)\geq -s+n|T|^\frac{1}{n}\cr
&\geq&-s+nt_0^\frac{1}{n}e^\frac{s}{n}\cr
&\geq&-s+nt_0^\frac{1}{n}\left(t_0^{-\frac{1}{n}}+\frac{t_0^{-\frac{1}{n}}}{n}(s+\log t_0)\right)\cr
&=&\log t_0 +n\cr
&=&\langle(-1,x_{I_n}),(-\log t_0, x_{I_n})\rangle.
\end{eqnarray*}
Consequently, the supporting hyperplane to $C_1$ at $(-\log t_0, x_{I_n})$ is orthogonal to the vector $(-1,x_{I_n})$. If $t_0=1$, then the supporting hyperplane at the point $(-\log t_0, x_{I_n})=(0,x_{I_n})$ is not unique. Notice that in such case, for any $a\geq -1$
\begin{eqnarray*}
\langle(a,x_{I_n}), (s,x_T) \rangle&=&as+\textrm{tr}(T)\geq as+n|T|^\frac{1}{n}\cr
&\geq&as+nt_0^\frac{1}{n}e^\frac{s}{n}\cr
&\geq&as+nt_0^\frac{1}{n}\left(1+\frac{1}{n}(s+1)\right)\cr
&=&(a+1)s +n\cr
&\geq&n\cr
&=&\langle(a,x_{I_n}),(0, x_{I_n})\rangle.
\end{eqnarray*}
and if $a<-1$ then there exist some $s>0$ such that $e^\frac{s}{n}<1-\frac{as}{n}$. Then
\begin{eqnarray*}
\langle (a,x_{I_n}), (s,e^\frac{s}{n}x_{I_n})\rangle&=&as+ne^\frac{s}{n}\cr
&<& n\cr
&=& \langle (a,x_{I_n}), (0,x_{I_n})\rangle.
\end{eqnarray*}
Thus, if $t_0=1$, a hyperplane orthogonal to a vector $(a,x_{I_n})$ through $(0,x_{I_n})$ is a supporting hyperplane to $C_1$ if and only if $a\geq -1$. Besides,
$$
P_{e_1^\perp}N(C_1,(0,x_{I_n}))=N_{e_1^\perp}(C_1\cap{e_1^\perp},(0,x_{I_n}))=\textrm{pos}\{-x_{I_n}\}.
$$
(The proof of the last inequality can be found in the proof of the geometric case in \cite{GS}). Thus all the supporting hyperplanes to $C_1$ at $(0,x_{I_n})$ are hyperplanes orthogonal to some vector $(a,x_{I_n})$ with $a\geq -1$.

Now, let us assume that $\mathcal{E}(f)=(B_2^n)^{t_0\Vert f\Vert_\infty}$. Then, since $\mathcal{E}(f)$ is unique, $(-\log t_0, x_{I_n})$ is the unique common point to $E$ and $C_1$. Since both sets are convex there exists a hyperplane through $(-\log t_0, x_{I_n})$ separating them. If $t_0\neq 1$ this hyperplane has to be orthogonal to the vector $(-1,x_{I_n})$. If $t_0=1$, this hyperplane is not necessarily unique but it has to be orthogonal to some vector $(a,x_{I_n})$ with $a\geq -1$ and for every $(s,x_T)\in E$,
$$
\langle (a,x_{I_n}),(s,x_T)\rangle=as+\textrm{tr}(T)
\leq n=\langle (a,x_{I_n}),(0,x_{I_n})\rangle.
$$
Thus, if a vector $(a,x_{I_n})$ verifies this condition for any $(s, x_T)\in E,$ so does $(-1,x_{I_n})$ and in any case, the vector $(-1,x_{I_n})$ belongs to $N(E,(-\log t_0, x_{I_n}))$, the normal cone to $E$ at $(-\log t_0, x_{I_n})$.

Notice that $$E\cap(\{-\log t_0\}\times L)=\{(-\log t_0,x_T)\,:\,T\in C_{\textrm{spd}}, T B_2^n\subseteq K_{t_0}\}=\{-\log t_0\}\times E_0,$$
where
\begin{eqnarray*}
E_0&=&\{x_T\in L\,:\,T\in C_{\textrm{spd}},TB_2^n\subseteq K_{t_0}\}\cr
&=&\{x_T\in L\,:\,T\in C_{\textrm{spd}},\langle Tu,v\rangle\leq h_{K_{t_0}}(v)\,\forall u,v\in S^{n-1}\}\cr
&=&\{x_T\in L\,:\,T\in C_{\textrm{spd}},\langle x_T,x_{uv^t}\rangle\leq h_{K_{t_0}}(v)\,\forall u,v\in S^{n-1}\}.
\end{eqnarray*}
Therefore $E_0$ is given by the intersection of the convex cone $C_{\textrm{spd}}$ with a family of halfspaces $H_{u,v}$ that change continuously with $u,v\in S^{n-1}$. Then the translation of the support cone $x_{I_n}+S(E_0,x_{I_n})$ is the intersection of $C_{\textrm{spd}}$ with the halfspaces that pass through $x_{I_n}$
$$
x_{I_n}+S(E_0,x_{I_n})=\{x_T\in L\,:\,\langle Tu,v\rangle\leq h_{K_{t_0}}(v)\,\forall u,v\in S^{n-1}\textrm{ s.t. }\langle u,v\rangle=h_{K_{t_0}}(v)\}.
$$
Since $B_2^n\subseteq K_{t_0}$, the condition $\langle u,v\rangle=h_{K_{t_0}}(v)$ only occurs when $u=v$ and $u\in S^{n-1}\cap\partial K_{t_0}$ and then
\begin{eqnarray*}
x_{I_n}+S(E_0,x_{I_n})&=&\{x_T\in L\,:\,T\in C_{\textrm{spd}},\langle Tu,u\rangle\leq h_{K_{t_0}}(u),\,\forall u\in S^{n-1}\cap\partial K_{t_0}\}\cr
&=&\{x_T\in L\,:\,T\in C_{\textrm{spd}},\langle x_T,x_{uu^t}\rangle\leq 1,\,\forall u\in S^{n-1}\cap\partial K_{t_0}\}\cr
\end{eqnarray*}
Then the dual cone of $S(E_0,x_{I_n})$ is
$$
N(E_0,x_{I_n})=\textrm{pos}\{x_{uu^t}\,:\,u\in S^{n-1}\cap\partial K_{t_0}\}
$$
and so
$$N_{e_1^\perp}\left((E\cap(-\log t_0e_1+e_1^{\perp}), (-\log t_0,x_{I_n})\right)=\{0\}\times\textrm{pos}\{x_{uu^t}\,:\,u\in S^{n-1}\cap\partial K_{t_0}\}.$$
Since, by (\ref{DualityConesSubspaces}),
$$
N_{e_1^\perp}\left((E\cap(-\log t_0e_1+e_1^{\perp}), (-\log t_0,x_{I_n})\right)=P_{e_1^\perp}N(E,(-\log t_0,x_{I_n}))
$$
we have, by Caratheodory's theorem, that for any vector $(\alpha,x_T)\in N(E,(-\log t_0,x_{I_n}))$ there exist some positive $\{\mu_j\}_{j=1}^{m^\prime}$ and some vectors $\{u_j\}_{j=1}^{m^\prime}$ in $S^{n-1}\cap\partial K_{t_0}$ with $1\leq m^\prime\leq\frac{n(n+1)}{2}$ such that
$$
(\alpha,x_{T})=\left(\alpha,\sum_{j=1}^{m^\prime}\mu_jx_{u_j u_j^t}\right).
$$
Now, by Caratheodory's theorem again, there exist some vectors $\{(\alpha_i,x_{T_i})\}_{i=1}^{m}\in  N(E,(-\log t_0,x_{I_n}))$ and some positive numbers $\{\lambda_i\}_{i=1}^{m}$, with $1\leq m\leq 1+\frac{n(n+1)}{2}$ such that
$$
(-1,x_{I_n})=\sum_{i=1}^{m}\lambda_i(\alpha_i,x_{T_i})=\sum_{i=1}^{m}\lambda_i(\alpha_i,\sum_{j=1}^{m^\prime(i)}\mu_{ij}x_{u_{ij} u_{ij}^t}).
$$

Equivalently,  there exist some positive $\{\lambda_i\}_{i=1}^m$, $\{\mu_{ij}\}_{j=1}^{m^\prime(i)}$, some vectors $\{u_{ij}\}_{j=1}^{m^\prime(i)}$ in $S^{n-1}\cap\partial K_{t_0}$ and some numbers $\{\alpha_i\}_{i=1}^m$ such that
$$
\sum_{i=1}^m\lambda_i\alpha_i=-1,
$$
$$
I_{n}=\sum_{i=1}^m\sum_{j=1}^{m^\prime(i)}\lambda_i\mu_{ij} u_i\otimes u_i,
$$
and for any $t\in (0,1]$, and any $1\leq i\leq m$ if $TB_2^n\subset K_t$
$$
-\alpha_i\log\left(\frac{t}{t_0}\right)+\sum_{j=1}^{m^\prime(i)}\mu_{ij}\langle x_T,x_{u_{ij}u_{ij}^t}\rangle\leq \sum_{j=1}^{m^\prime(i)}\mu_{ij}.
$$
Since for any vector $u$ and any symmetric positive definite $T$ we have that
$$
\langle Tu,u\rangle=\sum_{i,j=1}^nt_{ij}u_iu_j=\langle x_T,x_{uu^t}\rangle,
$$
the last inequality is the same as
$$
-\alpha_i\log\left(\frac{t}{t_0}\right)+\sum_{j=1}^{m^\prime(i)}\mu_{ij}\langle Tu_{ij},u_{ij}\rangle\leq \sum_{j=1}^{m^\prime(i)}\mu_{ij}.
$$

Finally, notice that since $r(K_t)B_2^n\subseteq K_t$, then if a vector $(\alpha,\sum_{j=1}^{m^\prime}x_{u_ju_j^t})$, with $u_j\in S^{n-1}$ belongs to the normal cone $N(E,(s_0,x_{I_n}))$, it has to verify that
$$
-\alpha\log\left(\frac{t}{t_0}\right)+r(K_t)\leq 1
$$
and so for any $t> t_0$
$$
\alpha\geq\frac{r(K_t)-1}{\log\left(\frac{t}{t_0}\right)}
$$
and for any $t<t_0$
$$
\alpha\leq\frac{r(K_t)-1}{\log\left(\frac{t}{t_0}\right)}
$$
Thus, $\alpha$ belongs to the interval
$$
\left[\left.t_0\frac{d^+}{dt}r(K_t)\right|_{t=t_0},\left.t_0\frac{d^-}{dt}r(K_t)\right|_{t=t_0}\right]
$$
and so all the $\alpha_i$ belong to this interval.

Now assume that 2 holds. Then, since $(B_2^n)^{t_0\Vert f\Vert_\infty}\leq f$ we have that $(-\log t_0, x_{I_n})\in E$ and since the vectors $u_{ij}\in\partial K_{t_0}\cap S^{n-1}$, $(-\log t_0, x_{\lambda I_n})\notin E$ for any $\lambda>1$. Thus $(-\log t_0, x_{I_n})\in \partial E$ and we can consider the normal cone of $E$ at $(-\log t_0, x_{I_n})$, $N(E, (-\log t_0, x_{I_n}))$. The conditions in 2 say that
$$
(-1,x_{I_n})=\sum_{i=1}^m\lambda_i(\alpha_i,\sum_{j=1}^{m^\prime(i)}\mu_{ij}x_{u_{ij}u_{ij}^t}),
$$
with $(\alpha_i,\sum_{j=1}^{m^\prime(i)}\mu_{ij}x_{u_{ij}u_{ij}^t})\in N(E, (-\log t_0, x_{I_n}))$ and so $(-1,x_{I_n})\in N(E, (-\log t_0, x_{I_n}))$. Indeed, for any $t\in(0,1]$ and $\mathcal{E}=TB_2^n$ such that $(-\log t, x_T)\in E$ we have that
\begin{eqnarray*}
\langle (-\log t,x_T),(-1,x_{I_n})\rangle&=&\sum_{i=1}^m\lambda_i(-\alpha_i\log t+\sum_{j=1}^{m^\prime(i)}\mu_{ij}\langle x_T,x_{u_{ij}u_{ij}^t}\rangle)\cr
&\leq&\sum_{i=1}^m\lambda_i(\sum_{j=1}^{m^\prime(i)}-\alpha_i\log t_0)+\sum_{i=1}^m\sum_{j=1}^{m^\prime(i)}\lambda_i\mu_{ij}\cr
&=&n+\log t_0\cr
&=&\langle (-\log t_0,x_{I_n}),(-1,x_{I_n})\rangle.
\end{eqnarray*}
Thus, the supporting hyperplane to $C_1$ at $(-\log t_0,x_{I_n})$ orthogonal to $(-1,x_{I_n})$ is also a supporting hyperplane to $E$ at $(-\log t_0,x_{I_n})$ and so this is the unique point in the intersection of $C_1$ and $E$. Thus, $\mathcal{E}(f)=(B_2^n)^{t_0\Vert f\Vert_\infty}$.
\end{proof}
We also have a version of this theorem when we do not assume $f$ to be even:
\begin{thm}\label{FunctionalJohnNonSymmetric}
Let $f:\R^n\to\R$ be an integrable log-concave function and $t_0\in(0,1]$. Assume that $(B_2^n)^{t_0\Vert f\Vert_\infty}\leq f$. Then the following are equivalent:
\begin{enumerate}
\item[1.] $\mathcal{E}(f)=(B_2^n)^{t_0\Vert f\Vert_\infty}$.
\item[2.] There exist
 \begin{itemize}
 \item  $\{u_{ij}\}_{j=1}\subseteq \partial K_{t_0}\cap S^{n-1}$, $1\leq i\leq n$, $1\leq j\leq  m^\prime(i)$
 \item  $\{\lambda_i\}_{i=1}^m$, $\{\mu_{ij}\}_{j=1}^{m^\prime(i)}$, with $\lambda_i,\mu_{ij}>0$ and
 \item  $\{\alpha_i\}_{i=1}^m$, with $\alpha_i\in\left[t_0\left.\frac{d}{dt^+}r(K_t)\right|_{t=t_0},t_0\left.\frac{d}{dt^-}r(K_t)\right|_{t=t_0}\right]$
 \end{itemize}
 with , $1\leq 1+\frac{n(n+3)}{2}$ and $n\leq m^\prime(i)\leq \frac{n(n+3)}{2}$ for any $1\leq i\leq m$, such that
$$
\sum_{i=1}^m\lambda_i\alpha_i=-1,
$$
$$
\sum_{i=1}^m\sum_{j=1}^{m^\prime(i)}\lambda_i\mu_{ij}u_{ij}=0
$$
$$
I_{n}=\sum_{i=1}^m\sum_{j=1}^{m^\prime(i)}\lambda_i\mu_{ij} u_{ij}\otimes u_{ij},
$$
and for any $t\in (0,1]$, and any $1\leq i\leq m$ if $c+TB_2^n\subseteq K_t(f)$
$$
-\alpha_i\log\left(\frac{t}{t_0}\right)+\sum_{j=1}^{m^\prime(i)}\mu_{ij}\langle c,u_{ij}\rangle+\sum_{j=1}^{m^\prime(i)}\mu_{ij}\langle Tu_{ij},u_{ij}\rangle\leq \sum_{j=1}^{m^\prime(i)}\mu_{ij}.
$$
\end{enumerate}
\end{thm}
\begin{proof}
The proof follows the same lines as that of Theorem \ref{FunctionalJohn}, but considering the convex sets in $\R\times\R^n\times L$
$$
E=\{(s,c,x_T)\in[0,+\infty]\times\R^{n}\times L\,:\,T\in C_{\textrm{spd}},(c+TB_2^n)^{e^{-s}\Vert f\Vert_\infty}\leq f\}
$$
and
$$
C_1=\{(s,c,x_T)\in[0,+\infty]\times\R^n\times L\,:\,T\in C_{\textrm{spd}}e^{-s}|T|\geq t_0\}.
$$
In this case the unique contact point between these sets will be $(-\log t_{0},0,x_{I_n})$, the normal vector to $C_1$ at it will be $(-1,0,x_{I_n})$, and the projection of the normal cone to $E$ at it onto $e_{1}^\perp$ will be
$$
P_{e_{1}^\perp}(N(E,(-\log t_0,0,x_{I_n}))=\{0\}\times\textrm{pos}\{(u,x_{uu^t})\,:\,u\in S^{n-1}\cap\partial K\}.
$$
\end{proof}

\begin{rmk}
From the proof of Theorem \ref{FunctionalJohn} and Theorem \ref{FunctionalJohnNonSymmetric} we deduce that if $\mathcal{E}(f)=\mathcal{E}^{t_0\Vert f\Vert_\infty}$, then necessarily $t_0\geq e^{-n}$. Indeed, assume that $\mathcal{E}(f)=(B_2^n)^{t_0\Vert f\Vert_\infty}$. Since $E$ is convex, for any $\lambda\in(0,1)$, $\lambda(-\log t_0,x_{I_n})\in E$, (or $\lambda(-\log t_0,0,x_{I_n})+(1-\lambda)(0,a,0)$ for some a in the non-symmetric case). Then
$$
\langle\lambda(-\log t_0,x_{I_n}),(-1,x_{I_n})\rangle\leq \langle(-\log t_0,x_{I_n}),(-1,x_{I_n})\rangle
$$
or, equivalently,
$$
\lambda\log t_0+\lambda n\leq\log t_0+n,
$$
which implies that $\log t_0+n\geq0$, which is equivalent to $t_0\geq e^{-n}$.
\end{rmk}
\begin{cor}\label{EqualityInEllipsoidsEquation}
Let $f$ be an integrable log-concave function such that $\max_{t\in(0,1]}\phi_{f}(t)=\phi_{f}(t_0)$, {\it i.e.}, its John's ellipsoid is $\mathcal{E}(f)=\mathcal{E}_{t_0}(f)^{t_0\Vert f\Vert_\infty}$. Then for every $t\in(0,1]$
$$
|\mathcal{E}_t(f)|\leq\left(1-\log\left(\frac{t}{t_0}\right)^\frac{1}{n}\right)^n|\mathcal{E}_{t_0}(f)|.
$$
Besides, if there is equality for every $t\in(0,1]$, then for some $c_t\in\R^n$
$$
\mathcal{E}_t(f)=c_t+\left(1-\log\left(\frac{t}{t_0}\right)^\frac{1}{n}\right)\mathcal{E}_{t_0}(f).
$$
\end{cor}
\begin{proof}
By Lemma \ref{LinearMapsAndEllipsoids} we can assume that $\mathcal{E}(f)=(B_2^n)^{t_0\Vert f\Vert_\infty}$. Then by Theorem \ref{FunctionalJohnNonSymmetric} there exist
 \begin{itemize}
 \item  $\{u_{ij}\}_{j=1}\subseteq \partial K_{t_0}\cap S^{n-1}$, $1\leq i\leq m$, $n\leq j\leq  m^\prime(i)$
 \item  $\{\lambda_i\}_{i=1}^m$, $\{\mu_{ij}\}_{j=1}^{m^\prime(i)}$, with $\lambda_i,\mu_{ij}>0$ and
 \item  $\{\alpha_i\}_{i=1}^m$, with $\alpha_i\in\left[t_0\left.\frac{d}{dt^+}r(K_t)\right|_{t=t_0},t_0\left.\frac{d}{dt^-}r(K_t)\right|_{t=t_0}\right]$
 \end{itemize}
 with $1\leq m\leq 1+\frac{n(n+3)}{2}$ and $1\leq m^\prime(i)\leq \frac{n(n+3)}{2}$ for any $1\leq i\leq m$, such that
$$
\sum_{i=1}^m\lambda_i\alpha_i=-1,
$$
$$
\sum_{i=1}^m\sum_{j=1}^{m^\prime(i)}\lambda_i\mu_{ij}u_{ij}=0
$$
$$
I_{n}=\sum_{i=1}^m\sum_{j=1}^{m^\prime(i)}\lambda_i\mu_{ij} u_{ij}\otimes u_{ij},
$$
and for any $t\in (0,1]$, and any $1\leq i\leq m$ if $c_t+TB_2^n\subseteq K_t(f)$
$$
-\alpha_i\log\left(\frac{t}{t_0}\right)+\sum_{j=1}^{m^\prime(i)}\mu_{ij}\langle c_t,u_{ij}\rangle+\sum_{j=1}^{m^\prime(i)}\mu_{ij}\langle Tu_{ij},u_{ij}\rangle\leq \sum_{j=1}^{m^\prime(i)}\mu_{ij}.
$$
Multiplying the last inequality by $\lambda_i$ and summing in $i$ we obtain that for any $t\in (0,1]$ if $c_t+TB_2^n\subseteq K_t(f)$
$$
\log\left(\frac{t}{t_0}\right)+\textrm{tr} (T)\leq n
$$
and so
$$
\log\left(\frac{t}{t_0}\right)+n|T|^\frac{1}{n}\leq n.
$$
Thus, if $\mathcal{E}\subseteq K_t(f)$ then
$$
|\mathcal{E}|^\frac{1}{n}\leq\left(1-\log\left(\frac{t}{t_0}\right)^\frac{1}{n}\right)|\mathcal{E}_{t_0}(f)|^\frac{1}{n}.
$$
and so it happens for the John's ellipsoid of $K_t(f)$, $\mathcal{E}_t(f)$. Besides, if there is equality in this inequality there has to be equality in all the inequalities. Then $\mathcal{E}_t(f)=c_t+T_tB_2^n$ verifies that $tr(T)=n|T|^\frac{1}{n}$ and so it has to be a Euclidean ball. Thus
$$
\mathcal{E}_t(f)=c_t+\left(1-\log\left(\frac{t}{t_0}\right)^\frac{1}{n}\right)\mathcal{E}_{t_0}(f).
$$
\end{proof}
The maximizers of the integral ratio will be log-concave functions like the ones defined in the following lemma. Let us study some of their properties
\begin{lemma}\label{maximizers}
For any $t_0\geq e^{-n}$ and convex body $K\subseteq\R^n$ with $0\in K$, let
$$
f_{K,t_0}(x)=e^{-\max\{\Vert x\Vert_K-(n+\log t_0),0\}}.
$$
Then
\begin{itemize}
\item $K_t(f_{K,t_0})=\left(1-\log\left(\frac{t}{t_0}\right)^\frac{1}{n}\right)K_{t_0}(f_{K,t_0})$
\item $\mathcal{E}_t(f_{K,t_0})=\left(1-\log\left(\frac{t}{t_0}\right)^\frac{1}{n}\right)\mathcal{E}_{t_0}(f_{K,t_0})$
\item $\max_{t\in(0,1]}\phi_{f_{K,t_0}}(t)=\phi_{f_{K,t_0}}(t_0)$
\item $\displaystyle{I.rat(f_{K,t_0})=\frac{v.rat(K)}{t_0^\frac{1}{n}}\left(\int_0^1\left(1-\log\left(\frac{t}{t_0}\right)^\frac{1}{n}\right)^ndt\right)^\frac{1}{n}}$
\item $I.rat(f_{K,t_0})$ is decreasing in $t_0$ in the interval $[e^{-n},1]$.
\end{itemize}
\end{lemma}
\begin{proof}
Notice that $\Vert f_{K,t_0}\Vert_\infty=1$. Then, by definition of $K_t(f_{K,t_0})$
\begin{eqnarray*}
K_t(f_{K,t_0})&=&\{x\in\R^n\,:\,\max\{\Vert x\Vert_K-(n+\log t_0),0\}\leq-\log t\}\cr
&=&\left(n-\log\left(\frac{t}{t_0}\right)\right)K=n\left(1-\log\left(\frac{t}{t_0}\right)^\frac{1}{n}\right)K
\end{eqnarray*}
Consequently, for any $t\in(0,1]$
$$
K_t(f_{K,t_0})=\left(1-\log\left(\frac{t}{t_0}\right)^\frac{1}{n}\right)K_{t_0}(f_{K,t_0}).
$$
Then
$$
\mathcal{E}_t(f_{K,t_0})=\left(1-\log\left(\frac{t}{t_0}\right)^\frac{1}{n}\right)\mathcal{E}_{t_0}(f_{K,t_0})
$$
and
$$
\phi_{f_{K,t_0}}(t)=\frac{t}{t_0}\left(1-\log\left(\frac{t}{t_0}\right)^\frac{1}{n}\right)^n\phi_{f_{K,t_0}}(t_0).
$$
Since the function $g(x)=x(1-\log x)$ attains its maximum at $x=1$, $\phi_{f_{K,t_0}}(t)$ attains its maximum at $t=t_0$. Consequently
\begin{eqnarray*}
I.rat(f_{K,t_0})^n&=&\frac{1}{t_0|\mathcal{E}_{t_0}(f_{K,t_0})|}{\int_{\R^n}f_{K,t_0}(x)dx}\cr
&=&\frac{1}{t_0|\mathcal{E}_{t_0}(f_{K,t_0})|}{\int_0^1|K_t(f_{K,t_0})|dt}\cr
&=&\frac{|K_{t_0}(f_{K,t_0})|}{t_0|\mathcal{E}_{t_0}(f_{K,t_0})|}\int_0^1\left(1-\log\left(\frac{t}{t_0}\right)^\frac{1}{n}\right)^ndt\cr
&=&\frac{v.rat(K)^n}{t_0}\int_0^1\left(1-\log\left(\frac{t}{t_0}\right)^\frac{1}{n}\right)^ndt.
\end{eqnarray*}
Changing variables $t=t_0e^{-s}$ we have
$$
I.rat(f_{K,t_0})^n=v.rat(K)\int_{\log t_0}^{+\infty}\left(1+\frac{1}{n}s\right)^ne^{-s}ds,
$$
which is clearly decreasing in $t_0\in[e^{-n},1]$.
\end{proof}
Now, we have the following, which in particular, since $I.rat(f_{B_\infty^n,t_0})$ and $I.rat(f_{\Delta^n,t_0})$ decrease in $t_0$, implies Theorem \ref{TheoremMaximizer}.
\begin{thm}
Let $t_0\in(0,1]$ and let $f:\R^n\to\R$ be an integrable log-concave  such that $\max_{t\in(0,1]}\phi_f(t)=\phi_f(t_0)$, {\it i.e.}, its John's ellipsoid is $\mathcal{E}(f)=\mathcal{E}_{t_0}(f)^{t_0\Vert f\Vert_\infty}$. Then we have that
$$
I.rat(f)\leq I.rat(f_{\Delta^n,t_0})
$$
with equality if and only if $\frac{f}{\Vert f\Vert_\infty}=f_{\Delta_n-c,t_0}\circ T$ for some affine map $T$ and some $c\in \Delta^n$.
If $f$ is even
$$
I.rat(f)\leq I.rat(f_{B_\infty^n,t_0})
$$
with equality if and only if $\frac{f}{\Vert f\Vert_\infty}=f_{B_\infty^n,t_0}\circ T$ for some $T\in GL(n)$.
\end{thm}
\begin{proof}
Let  $f:\R^n\to\R$ be such that $\max_{t\in(0,1]}\phi_f(t)=\phi_f(t_0)$. Then
\begin{eqnarray*}
I.rat(f)^n&=&\frac{1}{t_0\Vert f\Vert_\infty|\mathcal{E}_f(t_0)|}{\int_{\R^n}f(x)dx}
\cr
&=&\frac{1}{t_0|\mathcal{E}_f(t_0)|}{\int_0^1|K_t(f)|dt}\cr
&=&\frac{1}{t_0|\mathcal{E}_f(t_0)|}{\int_0^1v.rat(K_t)^n|\mathcal{E}_f(t)|dt}
\cr
&\leq&\frac{v.rat(\Delta^n)^n}{t_0|\mathcal{E}_f(t_0)|}{\int_0^1|\mathcal{E}_f(t)|dt}\cr
&\leq&\frac{v.rat(\Delta^n)^n}{t_0}\int_0^1\left(1-\log\left(\frac{t}{t_0}\right)^\frac{1}{n}\right)^ndt\cr
&=&I.rat(f_{\Delta^n,t_0})^n.
\end{eqnarray*}
Besides, if there is equality, all the inequalities are equalities and so $v.rat(K_t)=v.rat(\Delta^n)$, which implies that $K_t=T_t\Delta^n$, for some affine map $T_t$ and
$|\mathcal{E}_f(t)|=\left(1-\log\left(\frac{t}{t_0}\right)^\frac{1}{n}\right)^n|\mathcal{E}_f(t_0)|$, which by Corollary \ref{EqualityInEllipsoidsEquation} implies that the John's ellipsoid of every level set $\mathcal{E}_f(t)=c_t+\left(1-\log\left(\frac{t}{t_0}\right)^\frac{1}{n}\right)\mathcal{E}_f(t_0)$ and so $T_t=c_t+\left(1-\log\left(\frac{t}{t_0}\right)^\frac{1}{n}\right)T$ for every $t\in(0,1]$. Thus, we have that $$K_t=c_t+\left(1-\log\left(\frac{t}{t_0}\right)^\frac{1}{n}\right)T\Delta^n.$$

By Lemma \ref{LinearMapsAndEllipsoids} we can assume without loss of generality that $K_{t_0}=n\Delta^n$. In such case $c_{t_0}=0$. Then, calling $t=e^{-s}$ and $t_0=e^{-s_0}$ we have that
$$
K_{e^{-s}}=c_{e^{-s}}+\left(1+\frac{s}{n}-\frac{s_0}{n}\right)K_{e^{-s_0}}.
$$

By log-concavity, we have that for every $s\in[0,s_0]$
\begin{eqnarray*}
K_{e^{-s}}&\supseteq&\frac{s}{s_0}K_{e^{-s_0}}+\left(1-\frac{s}{s_0}\right)K_1\cr
&=&\left(1-\frac{s}{s_0}\right)c_1+\left(1+\frac{s}{n}-\frac{s_0}{n}\right)K_{e^{-s_0}}
\end{eqnarray*}
and then $c_{e^{-s}}=\left(1-\frac{s}{s_0}\right)c_1$. If $s\geq s_0$ we have that
\begin{eqnarray*}
K_{e^{-s_0}}&\supseteq&\frac{s_0}{s}K_{e^{-s}}+\left(1-\frac{s_0}{s}\right)K_1\cr
&=&\frac{s_0}{s}c_{e^{-s}}+\left(1-\frac{s_0}{s}\right)c_1+K_{e^{-s_0}}
\end{eqnarray*}
and also in this case $c_{e^{-s}}=\left(1-\frac{s}{s_0}\right)c_1$. Thus, for any $s\geq 0$
$$
K_{e^{-s}}=\left(1-\frac{s}{s_0}\right)c_1+\left(1+\frac{s}{n}-\frac{s_0}{n}\right)K_{e^{-s_0}}.
$$

Consequently, $\frac{f}{\Vert f\Vert_\infty}=e^{-v(\cdot)}\circ T$ with $T$ an affine map and
\begin{eqnarray*}
v(x)&=&\inf\left\{s\,:\,x\in K_{e^{-s}}\right\}\cr
&=&\inf\left\{s\,:\,x\in \left(1-\frac{s}{s_0}\right)c_1+\left(1+\frac{s}{n}-\frac{s_0}{n}\right)K_{e^{-s_0}}\right\}\cr
&=&\inf\left\{s\,:\,x\in\frac{n}{s_0}c_1+(s+n-s_0)\left(\frac{K_{e^{-s_0}}}{n}-\frac{1}{s_0}c_1\right)\right\}\cr
&=&\inf\left\{s\,:\,x-\frac{n}{s_0}c_1\in+(s+n-s_0)\left(\frac{K_{e^{-s_0}}}{n}-\frac{1}{s_0}c_1\right)\right\}\cr
&=&\max\left\{\left\Vert x-\frac{n}{s_0}c_1\right\Vert_{\left(\frac{1}{n}{K_{e^{-s_0}}}-\frac{1}{s_0}c_1\right)}-(n-s_0),0\right\}\cr
&=&\max\left\{\left\Vert x+\frac{n}{\log t_0}c_1\right\Vert_{\left(\frac{1}{n}{K_{t_0}}+\frac{1}{\log t_0}c_1\right)}-(n+\log t_0),0\right\}\cr
&=&\max\left\{\left\Vert x+\frac{n}{\log t_0}c_1\right\Vert_{\left(\Delta^n+\frac{1}{\log t_0}c_1\right)}-(n+\log t_0),0\right\}.\cr
\end{eqnarray*}
Notice that for $v$ to be well defined necessarily $c=\frac{1}{-\log t_0}c_1\in\Delta^n$ and then there exists $c\in\Delta^n$ such that
$$
\frac{f}{\Vert f\Vert_\infty}=e^{-\max\left\{\Vert \cdot-nc\Vert_{\left(\Delta^n-c\right)}-(n+\log t_0),0\right\}}\circ T
$$
or, equivalently,
$$
\frac{f}{\Vert f\Vert_\infty}=e^{-\max\left\{\Vert \cdot\Vert_{\left(\Delta^n-c\right)}-(n+\log t_0),0\right\}}\circ T
$$
The same proof works in the even case. In the even case we know that $c_t=0$ for any $t$ and then $T_t=\left(1-\log\left(\frac{t}{t_0}\right)^\frac{1}{n}\right)T$. Thus, we can assume  without loss of generality that $K_{t_0}=nB_\infty^n$ and then it implies that
$\frac{f}{\Vert f\Vert_\infty}=f_{B_\infty^n,t_0}\circ T$.
\end{proof}

Finally, we will compute the integral ratio of this maximizing function in the following lemma. We will do it for a whole class of functions that include the maximizing one.

\begin{lemma}
Let $\alpha\geq1$ and $f(x)=e^{-\Vert x\Vert_K^\alpha}$. Then
$$
I.rat(f)=\left(\frac{e\alpha\Gamma\left(1+\frac{n}{\alpha}\right)^{\frac{\alpha}{n}}}{n}\right)^{\frac{1}{\alpha}}v.rat(K)\sim v.rat(K).
$$
\end{lemma}
\begin{proof}
On one hand
\begin{eqnarray*}
\int_{\R^n}e^{-\Vert x\Vert_K^\alpha}dx&=&\int_{\R^n}\int_{\Vert x\Vert_K^\alpha}^{+\infty} e^{-t}dtdx=\int_0^{+\infty}\int_{t^\frac{1}{\alpha}K}e^{-t}dxdt\cr
&=&|K|\int_0^{+\infty} t^\frac{n}{\alpha}e^{-t}dt=|K|\Gamma\left(1+\frac{n}{\alpha}\right).
\end{eqnarray*}
On the other hand, for any $t\in(0,1]$
$$
K_t=(-\log t)^\frac{1}{\alpha}K
$$
and then,
$$
\mathcal{E}_t(f)=(-\log t)^\frac{1}{\alpha}\mathcal{E}(K),
$$
where $\mathcal{E}(K)$ is the John ellipsoid of $K$. Thus,
$$
\phi_f(t)=t(-\log t)^\frac{n}{\alpha}|\mathcal{E}(K)|.
$$
Let us find $\max_{t\in(0,1]}t(-\log t)^\frac{n}{\alpha}|\mathcal{E}(K)|=\max_{s\in[0,+\infty)}e^{-s}s^\frac{n}{\alpha}|\mathcal{E}(K)|$. Taking derivatives we obtain that this maximum is attained at $s=\frac{n}{\alpha}$ and so
$$
\max_{t\in(0,1]}t(-\log t)^\frac{n}{\alpha}|\mathcal{E}(K)|=\left(\frac{n}{\alpha}\right)^\frac{n}{\alpha}e^{-\frac{n}{\alpha}}|\mathcal{E}(K)|.
$$
Consequently
$$
I.rat(f)=\left(\frac{e\alpha\Gamma\left(1+\frac{n}{\alpha}\right)^{\frac{\alpha}{n}}}{n}\right)^{\frac{1}{\alpha}}v.rat(K).
$$
\end{proof}
\section{Reverse Sobolev-type inequalities}\label{ReverseSobolev}
In this section we will prove Theorem \ref{FunctionalReversePetty}. First we will define the polar projection body of a function
\begin{proposition}
Let $f:\R^n\to[0,+\infty)$ be a log-concave integrable function. If the following quantity is finite for every $x\in\R^n$ then it defines a norm
$$
\Vert x\Vert=2|x|\int_{x^\perp}\max_{s\in\R}f\left(y+s\frac{x}{|x|}\right)dy.
$$
Besides, if $f\in W^{1,1}(\R^n)$ this norm equals
$$
\Vert x\Vert=\int_{\R^n}|\langle\nabla f(y),x\rangle| dy.
$$
The unit ball of this norm is the polar projection body of $f$, which will be denoted by $\Pi^*(f)$.
\end{proposition}
\begin{proof}
Notice that
\begin{eqnarray*}
\Vert x\Vert&=&2|x|\int_{x^\perp}\max_{s\in\R}f\left(y+s\frac{x}{|x|}\right)dy\cr
&=&2|x|\int_0^{+\infty}\left|\left\{y\in x^\perp\,:\,\max_{s\in\R}f\left(y+s\frac{x}{|x|}\right)\geq t\right\}\right|dt\cr
&=&2|x|\Vert f\Vert_\infty\int_0^1\left|P_{x^\perp}K_t\right|dt\cr
&=&2\Vert f\Vert_\infty\int_0^1\Vert x\Vert_{\Pi^*(K_t)}dt
\end{eqnarray*}
and it is clear that it is a norm.

If $f\in W^{1,1}(\R^n)$, for almost every $t$ the boundary of $K_t$ is $\{x\in\R^n\,:\,f(x)=t\Vert f\Vert_\infty\}$  and we have
\begin{eqnarray*}
\Vert x\Vert_{\Pi^*(f)}&=&2|x|\Vert f\Vert_\infty\int_0^1\left|P_{x^\perp}K_t\right|dt\cr
&=&|x|\int_0^{\Vert f\Vert_\infty}\int_{\{f(x)=t\}}\left|\left\langle\nu(y),\frac{x}{|x|}\right\rangle\right|dH_{n-1}(y)dt
\end{eqnarray*}
where $\nu(y)$ is the outer normal unit vector to $\{x\in\R^n\,:\,f(x)\geq t\}$ and $dH_{n-1}$ is the Haussdorff measure on the boundary of it. Since $\nu(y)=\frac{\nabla f(y)}{|\nabla f(y)|}$ almost everywhere the above expression is
$$
\int_0^{\Vert f\Vert_\infty}\int_{\{f(x)=t\}}\left|\left\langle\frac{\nabla f(y)}{|\nabla f(y)|},x\right\rangle\right|dH_{n-1}(y)dt
$$
which, by the co-area formula, equals
$$
\int_{\R^n}|\langle\nabla f(y),x\rangle| dy.
$$
\end{proof}

We will use the following lemma to prove Theorem \ref{FunctionalReversePetty}.
\begin{lemma}\label{AsplundProductEuclideanBalls}
Let $f:\R^n\to\R$ be a log-concave function and $g(x)=(B_2^n)^a(x)$. Then
$$
\lim_{\varepsilon\to0^+}f\star g_\varepsilon(z)=f(z)
$$
and
$$
\lim_{\varepsilon\to0^+}\frac{f\star g_\varepsilon(z)-f(z)}{\varepsilon}=|\nabla f(z)|+f(z)\log a\quad \textrm{almost everywhere}.
$$
\end{lemma}
\begin{proof}
By definition of the Asplund product, since $f$ is continuous,
$$
\lim_{\varepsilon\to0^+}f\star g_\varepsilon(z)=\lim_{\varepsilon\to0^+}\sup_{z=x+y}f(x)a^\varepsilon\chi_{B_2^n}\left(\frac{y}{\varepsilon}\right)=\lim_{\varepsilon\to0^+}\sup_{y\in B_2^n}f(z-\varepsilon y)a^\varepsilon=f(z).
$$
Besides, if $f$ is differentiable in $z$,
\begin{eqnarray*}
\lim_{\varepsilon\to0^+}\frac{f\star g_\varepsilon(z)-f(z)}{\varepsilon}&=&\lim_{\varepsilon\to0^+}\sup_{y\in B_2^n}\frac{f(z-\varepsilon y)a^\varepsilon-f(z)a^\varepsilon+f(z)a^\varepsilon-f(z)}{\varepsilon}\cr
&=&\lim_{\varepsilon\to0^+}\sup_{y\in B_2^n}\frac{f(z-\varepsilon y)a^\varepsilon-f(z)a^\varepsilon}{\varepsilon}+f(z)\lim_{\varepsilon\to0^+}\frac{a^\varepsilon-1}{\varepsilon}.\cr
\end{eqnarray*}
Since
$$
\lim_{\varepsilon\to0^+}\sup_{y\in B_2^n}\frac{f(z-\varepsilon y)-f(z)}{\varepsilon}=|\nabla f(z)|,
$$
the previous limit equals $|\nabla f(z)|+f(z)\log a$.
\end{proof}

The following lemma was proved in \cite{CF}. We reproduce it here for the sake of completeness:
\begin{lemma}\label{AsplundProductSameFunction}
Let $f:\R^n\to\R$ be an integrable log-concave function. Then
\begin{eqnarray*}
\lim_{\varepsilon\to0^+}\frac{\int_{\R^n}f\star f_\varepsilon(x)dx-\int_{\R^n}f(x)dx}{\varepsilon}&=&n\int_{\R^n}f(x)dx+\int_{\R^n}f(x)\log f(x)dx\cr
\end{eqnarray*}
\end{lemma}
\begin{proof}
First of all, notice that if $f(x)=e^{-u(x)}$ with $u$ a convex function, then
$$
f\star f_\varepsilon(z)=e^{-(1+\varepsilon) u\left(\frac{z}{1+\varepsilon}\right)},
$$
since, as $u$ is convex, its epigraph $\textrm{epi}\,u$ is a convex set and then
\begin{eqnarray*}
\inf_{z=x+y} u(x)+\varepsilon u\left(\frac{y}{\varepsilon}\right)&=&\inf_{z=x+\varepsilon y} u(x)+\varepsilon u(y)\cr
&=&\inf\{\mu\,:\,(z,\mu)\in (1+\varepsilon)\textrm{epi}\,u\}\cr
&=&(1+\varepsilon)u\left(\frac{z}{1+\varepsilon}\right).
\end{eqnarray*}
Then,
\begin{eqnarray*}
\frac{\int_{\R^n}f\star f_\varepsilon(x)dx-\int_{\R^n}f(x)dx}{\varepsilon}&=&\frac{1}{\varepsilon}\left((1+\varepsilon)^n\int_{\R^n}e^{-(1+\varepsilon)u(x)}dx-\int_{\R^n}e^{-u(x)}dx\right)\cr
&=&\left(\frac{(1+\varepsilon)^n-1}{\varepsilon}\right)\int_{\R^n}e^{-(1+\varepsilon)u(x)}dx\cr
&+&\int_{\R^n}e^{-u(x)}\left(\frac{e^{-\varepsilon u(x)}-1}{\varepsilon}\right)dx.
\end{eqnarray*}
Now, taking limit when $\varepsilon$ tends to 0 we obtain the result. The monotone convergence theorem and possibly a translation of the function $u$ allows us to interchange limits.
\end{proof}

Now we are able to prove Theorem \ref{FunctionalReversePetty}:
\begin{proof}
Since all the quantities in the statement of the theorem are affine invariant, {\it i.e.}, they take the same value for $f$ and for $f\circ T$, we can assume that $f$ is in John's position. That is, $\mathcal{E}(f)=(B_2^n)^{t_0\Vert f\Vert_\infty}$. On the one hand, by Jensen's inequality
\begin{eqnarray*}
|\Pi^*(f)|^\frac{1}{n}&=&|B_2^n|^\frac{1}{n}\left(\int_{S^{n-1}}\left(\int_{\R^n}|\langle\nabla f(z),\theta\rangle|dz\right)^{-n}d\sigma(\theta)\right)^\frac{1}{n}\cr
&\geq&|B_2^n|^\frac{1}{n}\left(\int_{S^{n-1}}\int_{\R^n}|\langle\nabla f(z),\theta\rangle|dzd\sigma(\theta)\right)^{-1}\cr
&=&|B_2^n|^\frac{1}{n}\left(\frac{2}{n}\frac{|B_2^{n-1}|}{|B_2^n|}\int_{\R^n}|\nabla f(z)|dz\right)^{-1}.
\end{eqnarray*}
On the other hand, let $g(x)=\mathcal{E}(f)(x)$. By Lemma \ref{AsplundProductEuclideanBalls}, we have that
\begin{eqnarray*}
|\nabla f(z)|+f(z)\log (t_0\Vert f\Vert_\infty)&=&\lim_{\varepsilon\to0^+}\frac{f\star g_\varepsilon(z)-f(z)}{\varepsilon}\cr
&\leq&\lim_{\varepsilon\to0^+}\frac{f\star f_\varepsilon(z)-f(z)}{\varepsilon}\cr
\end{eqnarray*}
By Lemma \ref{AsplundProductSameFunction}, integrating in $z\in\R^n$ we have that
\begin{eqnarray*}
\int_{\R^n}|\nabla f(z)|dz&+&\int_{\R^n}f(z)dz\log (t_0\Vert f\Vert_\infty)
\leq n\int_{\R^n}f(z)dz+\int_{\R^n}f(z)\log f(z)dz.
\end{eqnarray*}
Then
$$
\int_{\R^n}|\nabla f(z)|dz\leq n\int_{\R^n}f(z)dz+\int_{\R^n}f(z)\log \frac{f(z)}{t_0\Vert f\Vert_\infty}dz.
$$
Consequently, $\frac{\Vert f\Vert_{\frac{n}{n-1}}|\Pi^*(f)|^\frac{1}{n}}{\left(\frac{|B_2^n|}{2|B_2^{n-1}|}\right)}$ is bounded below by
$$
\left((t_0\Vert f\Vert_\infty)^\frac{1}{n} I.rat(f)\left[\left(\frac{\int_{\R^n}f(x)dx}{\int_{\R^n}f^\frac{n}{n-1}(x)dx}\right)^\frac{n-1}{n}+\frac{\int_{\R^n}f(x)\log\left(\frac{f(x)}{t_0\Vert f\Vert_\infty}\right)^\frac{1}{n}dx}{\Vert f\Vert_{\frac{n}{n-1}}\Vert f\Vert_1^\frac{1}{n}}\right]\right)^{-1}.
$$
Since $t_0\geq e^{-n}$, we can write $t_0=e^{-s_0 n}$ for some $s_0\in[0,1]$ and then
\begin{eqnarray*}
&&t_0^\frac{1}{n} \left[\left(\frac{\int_{\R^n}f(x)dx}{\int_{\R^n}f^\frac{n}{n-1}(x)dx}\right)^\frac{n-1}{n}+\frac{\int_{\R^n}f(x)\log\left(\frac{f(x)}{t_0\Vert f\Vert_\infty}\right)^\frac{1}{n}dx}{\Vert f\Vert_{\frac{n}{n-1}}\Vert f\Vert_1^\frac{1}{n}}\right]\cr
&=&e^{-s_0}\left[(1+s_0)\left(\frac{\int_{\R^n}f(x)dx}{\int_{\R^n}f^\frac{n}{n-1}(x)dx}\right)^\frac{n-1}{n}+\frac{\int_{\R^n}f(x)\log\left(\frac{f(x)}{\Vert f\Vert_\infty}\right)^\frac{1}{n}dx}{\Vert f\Vert_{\frac{n}{n-1}}\Vert f\Vert_1^\frac{1}{n}}\right].
\end{eqnarray*}
Since the maximum of $g(s)=e^{-s}\left[(1+s)A+B\right]$ with $A\geq0$, $B\leq 0$ and $s\geq 0$ is attained when $s=\frac{-B}{A}$ we have that $\frac{\Vert f\Vert_{\frac{n}{n-1}}|\Pi^*(f)|^\frac{1}{n}}{\left(\frac{|B_2^n|}{2|B_2^{n-1}|}\right)}$ is bounded below by
$$
\left(e^{\frac{\int_{\R^n}f(x)\log\left(\frac{f(x)}{\Vert f\Vert_\infty}\right)^\frac{1}{n}dx}{\int_{\R^n}f(x)dx}}\Vert f\Vert_\infty^\frac{1}{n}\left(\frac{\int_{\R^n}f(x)dx}{\int_{\R^n}f^\frac{n}{n-1}(x)dx}\right)^\frac{n-1}{n}I.rat(f)\right)^{-1}.
$$
\end{proof}

\end{document}